\documentclass[a4paper,11pt,reqno]{amsart}
\usepackage[left=3cm,right=3cm,top=3.5cm,bottom=3.5cm]{geometry}
\usepackage{amsmath,amssymb,amsthm,amscd,amsfonts}
\usepackage{latexsym}
\usepackage{graphicx}
\usepackage{booktabs} 
\usepackage[all]{xy}
\usepackage{xcolor}
\usepackage[OT2,T1]{fontenc}
\usepackage{hyperref}
\usepackage{mathtools}
\usepackage{tikz}
\usepackage{tikz-cd}
\usetikzlibrary{arrows}
\usepackage{float}
\usepackage{caption}
\usepackage[new]{old-arrows}
\usepackage{multirow}
\usepackage{array}
\allowdisplaybreaks

\DeclareMathOperator{\Pic}{Pic}

\DeclareMathOperator{\Hom}{Hom}
\DeclareMathOperator{\Gal}{Gal}
\DeclareMathOperator{\ev}{ev}

\theoremstyle{plain}
\newtheorem{theorem}{Theorem}[section]
\newtheorem{lemma}[theorem]{Lemma}
\newtheorem{proposition}[theorem]{Proposition}
\newtheorem{corollary}[theorem]{Corollary}

\numberwithin{equation}{section} \numberwithin{figure}{section}

\theoremstyle{definition}
\newtheorem{definition}[theorem]{Definition}

\newtheorem*{notation}{Notation}

\theoremstyle{remark}
\newtheorem{remark}[theorem]{Remark}
\numberwithin{table}{section}

\newcommand{\Z}{{\mathbb Z}}

\newcommand{\Q}{{\mathbb Q}}	
\newcommand{\R}{{\mathbb R}}
\newcommand{\F}{{\mathbb F}}
\newcommand\GG{\mathbb{G}}

\renewcommand{\P}{{\mathbb P}}	
\renewcommand{\H}{\mathrm{H}}
\newcommand{\Br}{{\mathrm{Br}\,}}
\newcommand{\Brzero}{{\mathrm{Br}_0\,}}
\newcommand{\Brone}{{\mathrm{Br}_1\,}}
\renewcommand{\bar}{\overline}

\newcommand{\inv}{\operatorname{inv}}
\newcommand{\A}{{\mathbb A}}
\newcommand{\legendre}[2]{\left(\frac{#1}{#2}\right)}
\renewcommand{\U}{{\mathcal{U}}}
\newcommand{\modn}{{\rm{mod} \hspace{0.1cm} }}

\renewcommand{\leq}{\leqslant}
\renewcommand{\geq}{\geqslant}

\title{Integral Brauer--Manin Obstruction for Fermat Near Misses}
\author[J. Alessandrì]{Jessica Alessandrì}
\address{Jessica Alessandrì, Max Planck Institute for Mathematics\\
Vivatsgasse 7\\
53111 Bonn\\
Germany}
\email{alessandri@mpim-bonn.mpg.de}
\urladdr{https://sites.google.com/view/jessicaalessandri}
\author[J. He]{Jiazhi He}
\address{Jiazhi He, Department of Mathematics, University of Bath, Bath, BA2 7AL, UK.}
\email{jh3671@bath.ac.uk, hjz19980623@gmail.com}


\date{\today}

\begin{document}

\begin{abstract}
    We study the integral Hasse principle for the affine surfaces given by $x^4+y^4-w^2=m^2$, where $m>1$ is odd and squarefree.
    We give a criterion to determine if there is an integral Brauer--Manin obstruction, which depends only on simple congruence conditions on the prime factors of $m$ modulo $8$. We use this to show that for $100\%$ of such surfaces there is no integral Brauer--Manin obstruction to the Hasse principle for integral points.
\end{abstract}
\maketitle

\tableofcontents

\section{Introduction}

The question of which integers can be represented as a sum of three cubes is a long-standing problem in number theory, which has been considered by many authors: \cite{HB92,Colliot2012Wittenberg,Boo19,BS21} to name a few. The more general case of integers represented by diagonal ternary cubic forms is considered in the recent paper \cite{Lyczak2023}, where it is shown that local representations are not enough to ensure integral representation: they show that the \emph{Hasse principle} do not hold, by studying the \emph{Brauer--Manin obstruction} on the affine surface defined by the given equation.

In this paper, we address the same question, for weighted diagonal quartics, i.e.\  equations
\begin{equation} \label{eq:fermat_near_miss}
  x^4 + y^4 - w^2 = n, \quad n \in \Z.
\end{equation}
As in the case of the sum of three cubes, one can see these equations as ``Fermat near misses'', because of the equation $x^4 + y^4 = w^2$, which helped to prove the quartic case of Fermat's Last Theorem.
These equations were recently studied in \cite{alessandriloughran}, where it was shown that, for many of these equations, once a ``good''  integer solution (in terms of some sign conditions on the coordinates) exists, then there are infinitely many and, more than that, there is is a \emph{not--thin} set of them.
However, the problem of determining if there is a solution at all is still open.

We focus on the case when \(n = m^2\) is a square number. Already in the case when \(n = 9 \) one can show that equation \eqref{eq:fermat_near_miss} has solutions everywhere locally, but we could not find integral solutions up to height $10^4$.

The main result of the paper is that the failure of the Hasse principle, explained by a Brauer--Manin obstruction, only depends on the primes dividing \(m\) and congruent to \(3\) or \(5\) modulo \(8\).
To be more precise, for $j\in\{1,3,5,7\}$ let
\[
  \omega_j(m)=\#\{q\mid m : q\ \text{prime},\ q\equiv j \pmod 8\}.
\]
Our main theorem is the following.
\begin{theorem}\label{thm:DP-BMO}
Let $m>1$ be odd and squarefree. Then there is an integral Brauer--Manin obstruction to the Hasse principle for \(x^4 + y^4 - w^2 = m^2\) if and only if \(\omega_5(m)=0\) and \(\omega_3(m)\) is odd.
\end{theorem}
An immediate consequence is that if \(m\) is an odd prime, then there is an integral Brauer--Manin obstruction if and only if \(m \equiv 3 \pmod 8\). This explains why we could not find integer solutions for \(n=9\), although it is everywhere locally soluble.

Theorem~\ref{thm:DP-BMO} also determines how often the obstruction occurs. Let $N_{\Br}(X)$ be the number of odd squarefree $m\leq X$ for which there is an integral Brauer--Manin obstruction to the Hasse principle for our equation. We show in Section~\ref{sec:counting} that
\[
  N_{\Br}(X)\sim c\,\frac{X}{(\log X)^{1/4}}
\]
for an explicit constant $c>0$ (see Corollary~\ref{cor:counting}). Since the number of odd squarefree $m\leq X$ is asymptotic to $\tfrac{4}{\pi^2}X$, we conclude that for $100\%$ of the odd squarefree $m$ there is no integral Brauer--Manin obstruction.

The proof of Theorem~\ref{thm:DP-BMO} relies on the computation of the Brauer group \(\Br U_{m^2}\) of the affine surface \(U_{m^2}\) defined by \(x^4 + y^4 - w^2 = m^2\).
We relate the Brauer group of \(U_{m^2}\) to the Brauer group of its natural compactification \(X_{m^2}\), which is a del Pezzo surface of degree $2$ in the weighted projective space $\P(1,1,1,2)$.
\begin{theorem}\label{thm:DP-affine-BG}
    For any odd squarefree integer \(m>1\) we have that the Brauer group of \(U_{m^2}\) and the Brauer group of \(X_{m^2}\) are isomorphic modulo constants, and moreover
    \[
      \Br U_{m^2}/\Br\Q=\Br X_{m^2}/\Br \Q \simeq(\Z/2\Z)^2.
    \]
\end{theorem}

We also compute the two explicit quaternion algebras $\alpha_1,\alpha_2$ (see Corollary~\ref{cor:generators}) generating $\Br U_{m^2}/\Br\Q$;
thus, to compute the Brauer set, we are reduced to computing the local invariants of $\alpha_1$ and $\alpha_2$.
By Lemma~\ref{lem:easy}, only the place $2$ and the primes dividing $m$ contribute, and we compute these in Section~\ref{sec:local}.

To show that the two Brauer groups are isomorphic, we adopt a strategy similar to that of \cite{Lyczak2023}, which generalizes \cite{uppal2024integral, Colliot2012Wittenberg} for affine diagonal cubic surfaces.
A key feature of our surface is that it admits a conic bundle structure, which makes the computations of the Brauer group and of the lines easier.
To compute the algebraic part of the Brauer group, we use \texttt{Magma} \cite{Magma}, running through all possible Galois actions compatible with this conic bundle structure.
One of the most delicate step in the proof of Theorem~\ref{thm:DP-affine-BG} is the computation of the transcendental Brauer group: it requires a delicate study of the isogeny class of the elliptic curve at infinity, namely the boundary divisor of the compactification $U_{m^2} \subset X_{m^2}$.
We then use the explicit form of the generators of the Brauer group to compute the local invariants, which leads to the proof of Theorem~\ref{thm:DP-BMO}.
Finally, we apply this result to count the number of $m$ for which there is a Brauer--Manin obstruction. For this last step, we follow an argument similar to that of \cite{LoughranMitankin2021}.

\subsection{Structure of the paper}
The paper is structured as follows.
In Section~\ref{sec:preliminaries} we recall the Brauer group, the integral Brauer--Manin obstruction, and some geometry of the affine surface \(U_{m^2}\) and its compactification \(X_{m^2}\).
Section~\ref{sec:brauer_groups} is devoted to the computation of the Brauer groups and generators.
In Section~\ref{sec:local} we compute the local invariants, then, in Section~\ref{sec:proof}, we prove Theorem~\ref{thm:DP-BMO}.
Finally, in Section~\ref{sec:counting} we show that for $100\%$ of the parameters $m$ there is no integral Brauer--Manin obstruction.

\begin{notation}
    We denote by \(K\) a number field and with \(\bar{K}\) a fixed algebraic closure.
    For \(X\) a variety over \(K\) and \(L\) a field extension of \(K\) we denote by \(X_L\) the base change of \(X\) over \(L\). We write \(\bar{X}\) for the base change of \(X\) to \(\bar{K}\).
    For $K=\Q$, we set the convention that at the real place $\Z_{\infty} = \R$. Moreover, we denote with $\A_\Z$ the integral adelic ring, given by the restricted product $\R \times \prod_{p \text{ prime}} \Z_p$.
    For $\mathcal{X}$ an integral model of a variety $X$ defined over $\Q$, we denote with $\mathcal{X}(\A_\Z)$ the set of integral adelic points, that is the restricted product $\mathcal{X}(\A_\Z) = \mathcal{X}(\R) \times \prod_{p \text{ prime}}\mathcal{X}(\Z_p)$.
\end{notation}

\subsection*{Acknowledgements}
The authors deeply thank Daniel Loughran for his constant support and valuable comments on a first draft of this paper, and for the help with the Magma code. We also thank Vladimir Mitankin and H.~Uppal for useful discussions and help.

JA is supported by the Max Planck Institute for Mathematics in Bonn, which she thanks for their hospitality and financial support. She is a member of the INdAM group GNSAGA. JH is supported by EPSRC.

\subsection*{Statement on the use of AI}
During the preparation of this work, the authors used Anthropic's Claude Opus 5.5 and OpenAI's GPT-5.5 to solely assist with script debugging of the code and for grammar checking and language editing.
All the mathematical content and ideas are of the authors, who take full responsibility for the final manuscript.

\section{Preliminaries}\label{sec:preliminaries}
In this section we briefly recall the Brauer group and the integral Brauer--Manin obstruction \cite{Colliot2009Xu} and the geometry of the affine surface \(U_n\) and its compactification \(X_n\).

\subsection{Integral Brauer--Manin Obstruction}
\begin{definition}
    For a variety \(X\) over \(K\), the \emph{Brauer group} of \(X\) is the second cohomology group
    \[ \Br X = \H_{\text{\'et}}^2(X, \GG_m).\]
    We define the groups \(\Brzero X = \text{Im}(\Br K \rightarrow \Br X)\) and \(\Brone X= \ker(\Br X \rightarrow\Br \bar{X})\).  They fit into the chain of inclusions
    \[ 0 \subseteq \Brzero X \subseteq \Brone X \subseteq \Br X. \]
    The group \(\Brone X\) is called the \emph{algebraic Brauer group} and the quotient \(\Br X/\Brone X\) is called the \emph{transcendental Brauer group of \(X\)}.
\end{definition}

Let \(U\) be a variety over \(\Q\) and fix \(\U\) an integral model over \(\Z\) of \(U\) given by the same equation.
For every \(\alpha \in \Br U\) and every place \(v\) of \(\Q\) we have an \emph{invariant map} \(\inv_v \alpha: \U(\Z_v) \rightarrow \Q/\Z\) and there exists \(S_\alpha\) a finite set of places such that \(\inv_v\alpha\) is zero for all \(v \notin S_\alpha\). Thus we have a well defined \emph{Brauer--Manin pairing}
\[
\begin{matrix}
    \U(\A_\Z) \times \Br U & \longrightarrow & \Q/\Z\\
    ((x_v)_v, \alpha) & \longmapsto & \sum_{v} \inv_v(\alpha(x_v)).
\end{matrix}
\]
\begin{definition}
    The left kernel of the above pairing is defined as the \emph{integral Brauer--Manin set}
    \[
    \U(\A_\Z)^{\Br} := \left\{ (x_v)_v \in \U(\A_\Z) \, \mid \, \sum_v\inv_v(\alpha(x_v)) = 0 \ \forall \alpha \in \Br U\right\}.
    \]
    We have the chain of inclusions
    \[ \U(\Z) \subseteq \U(\A_\Z)^{\Br} \subseteq \U(\A_\Z) \]
    and we say that there is an \emph{integral Brauer--Manin obstruction to the Hasse principle} if \( \U(\A_\Z)^{\Br} = \emptyset\) but \(\U(\A_\Z) \neq \emptyset\).
\end{definition}

\subsection{Geometry and conic bundle}
Let \(n\) be an integer and let \(U_n\) be the affine surface defined by
\[ U_n: \; x^4+y^4-n = w^2.\]
We denote by \(X_n\) the natural compactification of \(U_n\), i.e.\ the degree-\(2\) del Pezzo surface in the weighted projective space \(\P(1,1,1,2)\) defined by 
\[ X_n: \; x^4+y^4-nz^4 = w^2. \]
Thus \(U_n = X_n \setminus D\) where \(D\) is the divisor on \(X_n\) given by \(z=0\), which is an elliptic curve with only four $\Q$-rational points, given by \([1:0:0:\pm1]\), \([0:1:0:\pm1]\).

When \(n = m^2\) is a square number, as shown in \cite[Proposition 5.2]{alessandriloughran}, the del Pezzo surface $X_n$ has a conic bundle structure:
\[\begin{matrix}
\pi : &X_n &\longrightarrow &\mathbb{P}^1\\
&(x,y,z,w)&\longmapsto &(x^2-mz^2,\; w-y^2).
\end{matrix}
\]
The conic bundle morphism induces an embedding of $X_n$ in $\P^1\times\P^2$, where we take the variables $(s,t) \times (x_0,x_1,x_2)$ on $\P^1 \times \P^2$, with image given by
\begin{equation}\label{eq:conic_bundle}
    (s^2-t^2)x_0^2 - 2stx_1^2 + m(s^2+t^2)x_2^2=0.
\end{equation}
\begin{lemma}\label{lem:singular-fibres}
Let \(m\) be an integer and let \(\pi\colon X_{m^2}\to\mathbb P^1_{[s:t]}\) be the conic bundle inducing \eqref{eq:conic_bundle}. The singular fibres of \(\pi\) lie above the closed points \(P\) listed below, and the fibre above \(P\) splits over \(\Q(P)(\sqrt{a_P})\).
\begin{table}[H]
    \begin{center}
    \begin{tabular}{c|c|c|c}
    \(P\) & \(\Q(P)\) & singular fibre & \(a_P\)\\\hline
    \({[1:0]}\) & \(\Q\) & \(x_0^2+mx_2^2=0\) & \(-m\)\\
    \({[0:1]}\) & \(\Q\) & \(x_0^2-mx_2^2=0\) & \(m\)\\
    \({[1:1]}\) & \(\Q\) & \(x_1^2-mx_2^2=0\) & \(m\)\\
    \({[1:-1]}\) & \(\Q\) & \(x_1^2+mx_2^2=0\) & \(-m\)\\
    \(s^2+t^2=0\) & \(\Q(i)\) & \(x_0^2-ix_1^2=0\) & \(i\)
    \end{tabular}
    \caption{Singular fibres of the conic bundle}
    \label{tab:sing_fibres}
    \end{center}
\end{table}
\end{lemma}
\begin{proof}
The determinant of the diagonal form in \eqref{eq:conic_bundle} is \(-2mst(s-t)(s+t)(s^2+t^2)\), and substituting each root into \eqref{eq:conic_bundle} gives the displayed fibre.
\end{proof}
The boundary divisor $D$ corresponds to taking \(x_2=0\) in the conic bundle model. Hence by the above lemma (and see also \cite[proof of Lemma 5.4]{alessandriloughran}), all the four rational points of \(D\) are singular points of the singular fibres. In particular, they are the intersection points of the two Galois conjugate lines in each singular fibre.

\section{Brauer Group of Fermat near misses}\label{sec:brauer_groups}
The Brauer group of \(X_n\) is studied in \cite{Kresch2004On2} and, using the \texttt{Magma} computations accompanying the paper, one finds that \(\Br X_n/\Br\Q\) is isomorphic to \((\Z/2\Z)^2\), in the case when $n$ is a square $m^2$, with $m\neq 1$.
We recompute this using \cite[Corollary 11.3.5]{ColliotSkoro2021}, finding explicit generators that we will need later.
\begin{corollary}\label{cor:generators}
Let \(m>1\) be odd and squarefree. Then \(\Br X_{m^2}/\Br \Q\) is isomorphic to \(\Z/2\Z \times \Z/2\Z\) and is generated by the classes of the quaternion algebras
\begin{equation}\label{eq:brauer-generators}
 \alpha_1=\left(\frac{s}{s-t},m\right),\qquad
 \alpha_2=\left(\frac{t}{s+t},-m\right),\qquad
\end{equation}
where \(s=x^2-m\) and \(t=w-y^2\), written on the affine chart \(z=1\).
\end{corollary}
\begin{proof}
This is the standard computation of Prop.~11.3.4 and Cor.~11.3.5 in \cite{ColliotSkoro2021} applied to Lemma~\ref{lem:singular-fibres}. We use the same notation as in \cite{ColliotSkoro2021}, where \(S\) denotes the finite set of closed points \(P\in\mathbb P^1\) above which the fibre of the conic bundle is singular.

The generic conic class does not come from \(\Br\Q\), as there is a non-split fibre over $s^2+t^2 = 0$.
A vector $\epsilon=(\epsilon_0,\epsilon_\infty,\epsilon_1,\epsilon_{-1},\epsilon_i)\in\F_2^5$ indexed by the five singular fibres defines a class if and only if it satisfies the norm equation
\begin{equation}
    \prod_{P\in S}N_{\Q(P)/\Q}(a_P)^{\epsilon_P}=1 \quad\text{in }\Q^\times/\Q^{\times 2}.
\end{equation}
In our case, this translates to $(-m)^{\epsilon_0+\epsilon_{-1}}m^{\epsilon_\infty+\epsilon_1}=1$ in $\Q^\times/\Q^{\times2}$, i.e.\ $\epsilon_0=\epsilon_{-1}$ and $\epsilon_\infty=\epsilon_1$. Quotienting by the all-ones relation from the generic conic class (which maps to $0$ in $\Br X_{m^2} $) gives a two-dimensional group, hence
\[
  \Br X_{m^2}/\Br\Q\cong (\mathbb Z/2\Z)^2.
\]
The vectors selecting $\{[0:1],[1:1]\}$ and $\{[1:0],[1:-1]\}$ give $\alpha_1$ and $\alpha_2$, respectively, and the degree-two fibre corresponds to $\alpha_1+\alpha_2$.
\end{proof}

We will later use this explicit description to show that there is an integral Brauer--Manin obstruction.

We can now move to the computation of the Brauer group of the affine surface $U_n$. Our goal is to show Theorem \ref{thm:DP-affine-BG}.
We first compute the algebraic Brauer group of \(U_{m^2}\) and show that it is isomorphic to the Brauer group of \(X_{m^2}\).

We will need the following auxiliary lemma, which holds for general $n$.
We denote by $L$ the splitting field of $X_n$, i.e.\ the smallest normal number field $L$ such that $\Pic \bar X_n\simeq \Pic X_{n,L}$.
\begin{lemma}
    The Picard groups $\Pic \bar X_n$ and $\Pic \bar U_n$ are torsion free and we have a natural map
    $\Br X_n /\Brzero X_n \rightarrow \Brone U_n /\Brzero U_n$,
    with
    \[\Br X_n /\Brzero X_n \simeq \H^1(\Gal(L/\Q), \Pic \bar X_n)\]
    \[\Brone U_n /\Brzero U_n \simeq \H^1(\Gal(L/\Q), \Pic \bar U_n).\]
\end{lemma}
\begin{proof}
For any number field $K$, as in \cite[\S 3.1]{Lyczak2023}, by the Hochshild--Serre spectral sequence, we have the group isomorphisms
\[ \Br X_n /\Brzero X_n \simeq \H^1(K, \Pic \bar X_n)\]
and
\[\Brone U_n /\Brzero U_n \simeq \H^1(K, \Pic \bar U_n).\]
Let $L$ be the splitting field of $X_n$.
The Picard group $\Pic (X_{n,L})$ is a free $\Z$-module of rank $8$ and, since the anticanonical class $-K_{X_n}$ is primitive, $\Pic (U_{n,L})$ is also a free $\Z$-module, of rank $7$. The action of $\Gal(\bar\Q/L)$ on both the Picard groups is trivial, thus
$$\H^1(\Gal(\bar\Q/L), \Pic (X_{n,L})) = \Hom(\Gal(\bar\Q/L), \Pic (X_{n,L}))= 0,$$ 
where the last equality holds since $\Gal(\bar\Q/L)$ is a profinite group and $\Pic (X_{n,L})$ is torsion-free and the same holds for $\Pic (U_{n,L})$.
From the inflation-restriction sequence, it follows that
$$ \Br X_n/\Brzero X_n \simeq \H^1(\Q, \Pic (X_{n,L})) \simeq \H^1(\Gal(L/\Q), \Pic (X_{n,L})),$$
$$ \Brone U_n/\Brzero U_n \simeq \H^1(\Q, \Pic(U_{n,L}))  \simeq \H^1(\Gal(L/\Q), \Pic(U_{n,L})).$$

From the irreducibility of $D$, we have the short exact sequence
\[0 \rightarrow \mathbb{Z}\cdot[D] \rightarrow \Pic(\bar X_n) \rightarrow
\Pic(\bar U_n) \rightarrow 0, \]
which induces the long exact sequence
\begin{align*}
0 = &\H^1(\Gal(L/\Q),\Z) \rightarrow \H^1(\Gal(L/\Q),\Pic(\bar X_n)) \rightarrow\\
 &\rightarrow \H^1(\Gal(L/\Q),\Pic(\bar U_n)) \rightarrow \H^2(\Gal(L/\Q),\Z) \rightarrow \cdots,
\end{align*}
and so gives the desired natural map $\Br X_n /\Brzero X_n \rightarrow \Brone U_n /\Brzero U_n$.
\end{proof}
\begin{proposition}\label{prop:algebraic_brauer_group}
    Let $m > 1$ be an odd squarefree integer. Then the map $\Br X_{m^2} \rightarrow \Brone U_{m^2}$ is an isomorphism.
\end{proposition}
\begin{proof}
We first establish that the natural map in the above lemma is an isomorphism.

To compute the splitting field $L$ of $X_{m^2}$, one can find in \cite[\S~2]{Kresch2004On2} a complete description of the $56$ lines on $X_n$.
From it, under our assumption $n=m^2$, with \(m \neq 1\) odd squarefree, one can compute that the splitting field of $X_n$ is $L=\Q(\zeta_8, \sqrt[4]{n})$, where $\zeta_8$ is a fixed choice of a primitive $8$th root of unity and $\sqrt[4]{n}$ is a fixed choice of a $4$th root of $n$.
Then $L=\Q(\zeta_8, \sqrt{m})$, which has Galois group over $\Q$ isomorphic to $(\Z/2\Z)^3$.

To compute the cohomology groups $\H^1(\Gal(L/\Q), \Pic \bar X_{m^2})$ and $\H^1(\Gal(L/\Q), \Pic \bar U_{m^2})$, we look at the action of the absolute Galois group $G_\Q:=\Gal(\bar\Q/\Q)$ on the respective Picard groups. The group $G_\Q$ acts on $\Pic\bar X_{m^2}$ preserving intersection pairings and the anticanonical divisor, hence we have a representation
$$ \rho: G_{\Q} \longrightarrow W(E_7),$$
where $W(E_7)$ is the Weyl group of the root system $E_7$ (\cite[Ch.\ IV, \S25]{Manin1986Cubic}).

The action factors through $\Gal(L/\Q)$, since $\Gal(\bar\Q/L)$ acts trivially. Thus the Galois action is determined by the image $G$ of $\rho$ in $W(E_7)$, which is isomorphic to $\Gal(L/\Q) \simeq (\Z/2\Z)^3$.

With the help of \texttt{Magma}, we list all the possible subgroups $G$ of $W(E_7)$ isomorphic to $(\Z/2\Z)^3$, up to conjugacy. There are $8074$ conjugacy classes of subgroups in $W(E_7)$, of which $211$ are isomorphic to the desired group $(\Z/2\Z)^3$. To ease the analysis, we use the conic bundle structure of our surface.
As seen in Table \ref{tab:sing_fibres}, the conic bundle has $4$ singular non-split fibres over $\Q$. Each singular fibre is made of two lines intersecting in one rational point of the divisor $D$ and the Galois action swaps the lines.
Thus in our computation we can just consider the case where there are at least $4$ orbits of pairs of lines. Adding also the condition that $\Br X_{m^2}/\Br\Q$ has rank $2$, there are only $7$ cases left.
For each of these, we find that $\H^1(\Gal(L/\Q), \Pic(\bar X_{m^2}))$ and $\H^1(\Gal(L/\Q), \Pic(\bar U_{m^2}))$ are isomorphic. The \texttt{Magma} code used for the computations is available at \cite{GitHubRepo}.

We have the following diagram with exact rows
\begin{figure}[H]
    \centering
    \begin{tikzpicture}
        \node(a0) at (-5,2) {$0$};
        \node(a) at (-3,2) {$\Brzero X_{m^2} $};
		\node(b) at (0,2) {$\Br X_{m^2}$};
		\node(c) at (3,2) {$\dfrac{\Br X_{m^2}}{\Brzero X_{m^2} }$};
        \node(c0) at (5,2) {$0$};
        \node(d0) at (-5,0) {$0$};
		\node (d) at (-3,0) {$\Brzero U_{m^2} $};
		\node(e) at (0,0) {$\Brone U_{m^2}$};
		\node(f) at (3,0) {$\dfrac{\Brone U_{m^2} }{\Brzero U_{m^2} }$};
        \node(f0) at (5,0) {$0.$};

        \node at (-2.7,1) {$\simeq$};
        \node at (3.3,1) {$\simeq$};
            
		\draw[-stealth] (a0) --(a);	
		\draw[-stealth] (a) --(b);
		\draw[-stealth] (b) --(c);
        \draw[-stealth] (c) --(c0);
		\draw[-stealth] (a) --(d);
		\draw[-stealth] (b) --(e);
		\draw[-stealth] (c) --(f);
        \draw[-stealth] (d0) --(d);
		\draw[-stealth] (d) --(e);
		\draw[-stealth] (e) --(f);
        \draw[-stealth] (f) --(f0);
    \end{tikzpicture}
\end{figure}
By the Five Lemma, we also have an isomorphism between $\Br X_{m^2}$ and $\Brone U_{m^2} $.
\end{proof}

To complete the proof of Theorem \ref{thm:DP-affine-BG}, we follow a similar strategy to the one in \cite[\S 7.1]{uppal2024integral} and \cite{Colliot2012Wittenberg}.
We use the purity sequence from \cite[Theorem 3.7.1]{ColliotSkoro2021}
\begin{equation}\label{eq:purity}
 0 \rightarrow \Br X_{m^2} \rightarrow\Br U_{m^2} \xrightarrow{\partial_D} \H^1_{\text{\'et}}(D, \Q/\Z).
\end{equation}
As in the mentioned papers, a key point in the proof is the study of the isogeny class of the elliptic curve $D$. Indeed, Proposition 3.5 of \cite{Lyczak2023}, which we now recall for the reader's convenience, relates the torsion elements of the transcendental part of the Brauer group of $U_n$ to the existence of rational isogenies.
\begin{proposition}[{\cite[Proposition 3.5]{Lyczak2023}}] \label{prop:torsion_isogenies}
    Let \(F\) be a field of characteristic \(0\) and \(S\) a smooth, geometrically rational projective surface over \(F\). Denote by \(H\) a hyperplane over \(F\) and \(V\) the affine surface \(V := S\setminus H\). If \(C:= H \cap S\) is a smooth, genus one curve, then a class of order \(n\) in \(\Br V / \Brone V\) induces an \(F\)-rational \(n\)-isogeny \[J(C) \twoheadrightarrow  C'\] to an elliptic curve \(C'\) with an \(F\)-point of order \(n\).
\end{proposition}

On the elliptic curve \(D\) we can perform a change of variables, choosing the point \([0:1:0:1]\) as the identity element \(O_D\), so that \(D\) is in Weierstrass equation
\[ y^2 = x^3 -4x,\]
i.e.\ the curve \href{https://www.lmfdb.org/EllipticCurve/Q/64/a/3}{\texttt{64.a3}} in the LMFDB.
From the isogeny graph associated to this curve (see e.g.\ \cite[Table 5]{Chiloyan2021Isogeny}) we see that its isogeny class consists of \(4\) elliptic curves, and every cyclic isogeny have degree \(2\).
More precisely, the isogeny class of \(D\) is then given by the following.
\begin{table}[H]
    \centering
    \begin{tabular}{|c|c|c|}
    \hline
         Elliptic curve & Isogeny degree & Mordell-Weil group \\
         \hline
         \(D\) & 1 &  \(\Z/2\Z \times \Z/2\Z\) \\
         \(E_1: y^2 = x^3 +x\) & 2 & \(\Z/2\Z\) \\
         \(E_2: y^2 = x^3 -44x+112\) & 2 & \(\Z/4\Z\) \\
         \(E_3: y^2 = x^3 -44x -112\) & 2 & \(\Z/2\Z\)\\
    \hline
    \end{tabular}
    \caption{Isogeny class of \(D\) over $\Q$}
    \label{tab:isogeny_class}
\end{table}

The following auxiliary lemma shows that this table still holds over any quadratic field \(K=\Q(\sqrt{m})\) with $m > 2$ (i.e.\ when \(K\) is not contained in the splitting field).

\begin{lemma}\label{lem:auxiliary}
    Let \(K = \Q(\sqrt{m})\), with \(m\) a positive integer, \(m \neq 1,2\). Then
    \begin{itemize}
        \item[(a)] \(D\) has exactly three \(2\)-isogenies over \(K\);
        \item[(b)] For each \(i=1,2,3\), \(E_i(K)[2] \simeq \Z/2\Z\);
        \item[(c)] The only curve with \(K\)-rational \(4\)-torsion points is \(E_2\).
    \end{itemize}
\end{lemma}
\begin{proof}
The elliptic curve \(D\) has full rational \(2\)-torsion: the order \(2\)-points are obtained by setting \(x^3-4x=0\), and they correspond to the \(4\) rational points of \(D\). Then, over any number field \(K\), \(D\) has precisely three degree-\(2\) isogenies (one for each order \(2\) subgroup of \(D(K)[2]\)). This proves point (a).

The abscissas of the \(2\)-torsion points of the curves \(E_i\)'s are the roots of the polynomials \(x^3+x\), \(x^3-44x+112 = (x-4)(x^2+4x-28)\) and \(x^3-44x-112 = (x+4)(x^2-4x-28)\), respectively. Since \(m \neq 2\), it is then immediate that the only points defined over \(K\) are the ones corresponding to rational roots and hence rational points.
In \cite[\S 5]{Bandini2016DivisionFeilds2} one can find explicit formulas for the \(4\)-torsion points, in terms of the roots of the Weierstrass polynomial defining the curve.
It is then easy to see that the only \(4\)-torsion points defined over \(\Q(\sqrt{m})\) (and \(m \neq 2\)) are the ones with rational coordinates, namely the ones of \(2\)-torsion and the two rational \(4\)-torsion points of \(E_2\), which have coordinates \((x,y)\) equal to \((6,8)\), \((6,-8)\), proving points (b) and (c).
\end{proof}

We are now ready to prove Theorem \ref{thm:DP-affine-BG}.
\begin{proof}[Proof of Theorem \ref{thm:DP-affine-BG}]
    From Proposition~\ref{prop:torsion_isogenies} and Table~\ref{tab:isogeny_class}, the only possible torsion in $\Br U_{m^2}/\Br X_{m^2}$ is $1$ and $2$. Thus we now show that there is no $2$ torsion.
    Let \(\mathcal{A} \in \Br U_{m^2}\setminus \Br X_{m^2}\) of $2$ torsion. We want to derive a contradiction.

    By Proposition~\ref{prop:algebraic_brauer_group}, \(\Br X_{m^2} \simeq \Brone U_{m^2}\), so \(\mathcal{A} \notin \Brone U_{m^2} \).
    Let \(K = \Q(\sqrt{m})\), where \(n = m^2\) in the definition of \(X_{m^2}\) and \(U_{m^2}\), then the base change \(\mathcal{A}_K \notin \Brone U_{m^2,K} \).
    Since 
    \(\Br X_{m^2,K}  \subseteq \Brone U_{m^2,K}  \),
    we have that \(\mathcal{A}_K \notin \Br X_{m^2,K} \).

     Choose the points \(P_1 = [0:1:0:1]\), \(P_2 = [0:1:0:-1] \in D(K)\). The lines
    \begin{align*}
        L_1 &: x-\sqrt{m}z=0, \quad w-y^2=0,\\
        L_2 &: x-\sqrt{m}z=0, \quad w+y^2=0
    \end{align*}
    are defined over \(K\) and intersect \(D\) transversally at the points \(P_1\) and \(P_2\), respectively.

    For each \(i = 1,2\), we get the following commutative diagram, where the rows come from the purity sequence \eqref{eq:purity}:
    \begin{figure}[H]
        \centering
        \begin{tikzpicture}
            \node(a0) at (-6,2) {$0$};
            \node(a) at (-3,2) {\(\Br X_{m^2,K} \)};
    		\node(b) at (0,2) {\(\Br U_{m^2,K} \)};
    		\node(c) at (4,2) {\(\H^1_{\text{\'et}}(D_{K}, \Q/\Z)\)};
            \node(d0) at (-6,0) {\(0\)};
    		\node (d) at (-3,0) {\(\Br L_i\)};
    		\node(e) at (0,0) {\(\Br(L_i \setminus P_i)\)};
    		\node(f) at (4,0) {\(\H^1_{\text{\'et}}(K(P_i), \Q/\Z).\)};
    
            \node at (1.7,2.3) {\(\partial_D\)};
            \node at (1.7,0.3) {\(\partial_{P_i}\)};
            \node at (4.4, 1) {\(\ev_i\)};
                
    		\draw[-stealth] (a0) --(a);	
    		\draw[-stealth] (a) --(b);
    		\draw[-stealth] (b) --(c);
    		\draw[-stealth] (a) --(d);
    		\draw[-stealth] (b) --(e);
    		\draw[-stealth] (c) --(f);
            \draw[-stealth] (d0) --(d);
    		\draw[-stealth] (d) --(e);
    		\draw[-stealth] (e) --(f);
        \end{tikzpicture}
    \end{figure}

Assuming that \(\mathcal{A}_K \notin \Br X_{m^2,K} \), by exactness, we have that \(\partial_D(\mathcal A_K)\) is \emph{not} trivial in \(\H^1_{\text{\'et}}\left(D_K, \Q/\Z\right)\).
On the other hand, for the bottom row: since \(L_i\setminus P_i\) is isomorphic to \(\mathbb A^1_K\), then \(\Br(L_i \setminus P_i) \simeq \Br \mathbb A^1_K \simeq \Br K\) and furthermore \(\Br L_i \simeq \Br \P^1_K \simeq \Br K\). Thus the image of \(\mathcal A_K\) in \(\Br(L_i \setminus P_i)\) is constant.
By commutativity of the diagram, we get that $\partial_D(\mathcal A_K)$ lies in the kernel of $\ev=(\ev_1,\ev_2)$.

Let \(\psi:E \rightarrow D\) be a degree-\(2\) cover corresponding to \(\partial_D(\mathcal A_K)\).
For each \(i=1,2\), the fibre above each \(P_i\) is made of \(2\) distinct \(K\)-rational points.
Choosing a \(K\)-point \(Q\) in the fibre of \(P_1 = O_D\), makes \(\psi\) into a degree-\(2\) isogeny of elliptic curves.
With this choice, a point in the fibre of \(P_2\) must be of order \(4\): indeed, let \(Q_2\) be such a point, then clearly \(\psi(2Q_2)=2\psi(Q_2)=2P_2=O_D\), so \(2Q_2 \in \ker\psi\), which is a subgroup of \(E[2]\). Since every point in the fibre of \(P_1\) is \(K\)-rational, then \(\ker\psi = \psi^{-1}(P_1)\) is a subgroup of \(E(K)[2]\).
By Lemma~\ref{lem:auxiliary}, the elliptic curve \(E\) is one of the curves in Table~\ref{tab:isogeny_class} and \(E(K)[2] \simeq \Z/2\Z\), so \(\ker\psi = E(K)[2]\).
If \(Q_2\) were \(2\)-torsion, since it is \(K\)-rational, it would be in \(\ker\psi\), and so \(\psi(Q_2) = O_D \neq P_2\). Thus \(Q_2\) is a \(K\)-rational point of order \(4\).

The only curve in the isogeny class of \(D\) having \(K\)-rational points of order \(4\) is \(E_2\) in Table \ref{tab:isogeny_class}. So we have that the only possibility for \(\psi\) is to be the unique (degree-\(2\)) isogeny between \(E_2\) and \(D\), which is
\[(x,y) \mapsto \left(\dfrac{x^2 - 4x + 4}{4(x - 4)}, \dfrac{x^2y -8xy + 12y}{8(x-4)^2}\right).\]
The \(K\)-rational torsion points of \(E_2\) are \(\{(6,-8),(4,0),(6,8), O_{E_2}\}\) and so one can compute that the only non-trivial $2$-torsion point in the image of $\psi$ is the point $(2,0)$.
Having chosen \(P_1\) as \(O_D\), then \(P_2=(0,0)\) and we see that there are no \(4\)-torsion points sent to \(P_2\).

We conclude that \(\partial_D(\mathcal{A}_K)\) must be trivial  and we must have that \(\mathcal{A}_K \in \Br X_{m^2,K} \).
This is a contradiction, as we took \(\mathcal{A}\) such that \(\mathcal{A}_K \notin \Br X_{m^2,K} \).
\end{proof}

\section{Local invariants}\label{sec:local}
In this section we perform local invariants calculations that will be used for the proof of Theorem~\ref{thm:DP-BMO}. 
From now on $m>1$ is odd and squarefree. Write $\boldsymbol{\alpha}=(\alpha_1,\alpha_2)\in\Br U_{m^2}\times\Br U_{m^2}$. We take the convention that the local invariants lie in $\mu_2$, rather than $\Z/2\Z$.
For a quaternion algebra $(a,b)$ over $\Q_v$ the local invariant is the Hilbert symbol $\inv_v(a,b)=(a,b)_v\in\{\pm1\}$, and the condition $\sum_v\inv_v\mathcal A(P_v)=0$ defining the Brauer--Manin set reads $\prod_v\inv_v\mathcal A(P_v)=1$. Thus for $P_v\in\U_{m^2}(\Z_v)$ we have $\boldsymbol{\alpha}(P_v)\in\Br\Q_v\times\Br\Q_v$ and
\[
  \inv_v\boldsymbol{\alpha}(P_v)=\bigl(\inv_v\alpha_1(P_v),\,\inv_v\alpha_2(P_v)\bigr)\in\{\pm1\}^2 .
\]
We briefly recall how to compute the Hilbert symbols.
For a unit $e\in\Z_q^\times$ we write $\legendre{e}{q}$ for the Legendre symbol of its reduction modulo $q$. 
For an odd prime $q$, writing $a=q^{A}a_0$, $b=q^{B}b_0$ with $a_0,b_0$ units,
\[
  (a,b)_q=(-1)^{AB(q-1)/2}\legendre{a_0}{q}^{B}\legendre{b_0}{q}^{A},
\]
where $\legendre{\cdot}{q}$ denotes the Legendre symbol.
For $q=2$, with $\epsilon(c)=\tfrac{c-1}{2}\bmod 2$ and $\omega(c)=\tfrac{c^2-1}{8}\bmod 2$,
\[
  (a,b)_2=(-1)^{\epsilon(a_0)\epsilon(b_0)+A\omega(b_0)+B\omega(a_0)} .
\]

Firstly, we compute the invariants at the real place and at odd primes not dividing $m$.

Recall that $s=x^2-m$ and $t=w-y^2$, and define $r=x^2+m$, $u=w+y^2$, so that $tu=sr$ on the chart $z=1$.

\begin{lemma}\label{lem:easy}
Let $P_v\in\U_{m^2}(\Z_v)$.
\begin{enumerate}
\item If $v=p$ is a prime with $p\nmid 2m$, then $\inv_p\boldsymbol{\alpha}(P_p)=(1,1)$.
\item At the real place $\inv_\infty\boldsymbol{\alpha}(P_\infty)=(1,1)$.
\end{enumerate}
\end{lemma}

\begin{proof}
(1) Let $\mathcal X/\mathbb Z_p$ be the smooth proper model for $X_{m^2}$ defined by the same weighted homogeneous equation. At the generic point of the special fibre, both entries of each quaternion representative in \eqref{eq:brauer-generators} are units, since $p \nmid m$ and none of $s,t,s+t,s-t$ vanishes identically on the special fibre.
So the vertical residues of $\alpha_1,\alpha_2$ are trivial; by purity for the regular scheme $\mathcal X$ both classes extend to $\Br\mathcal X$, and their evaluation at a $\mathbb Z_p$-point lies in $\Br\mathbb Z_p=0$.

(2) 
For $\alpha_1=(s/(s-t),m)$, as we are assuming $m>1$, the second entry is positive, so $\inv_\infty\alpha_1(P_\infty)=1$. For $\alpha_2=(t/(s+t),-m)$ the second entry is negative, and $t(u+r)=(s+t)r$ gives $t/(s+t)=r/(u+r)$ on $U_{m^2}$.
Since $(x^2+y^2+m)^2-(x^4+y^4-m^2)=2x^2y^2+2mx^2+2my^2+2m^2>0$, we have $r=x^2+m>0$ and $u+r=(x^2+y^2+m)+w>0$.
Hence $t/(s+t)>0$ on $U_{m^2}(\R)$ and $\inv_\infty\alpha_2(P_\infty)=1$.
\end{proof}

By Lemma~\ref{lem:easy}, only the place $2$ and the primes dividing $m$ can contribute. We study separately these cases in the following subsections.

As we will show in Section~\ref{sec:proof}, when the obstruction in Theorem~\ref{thm:DP-BMO} exists, it is given by the element $\beta=\alpha_1+\alpha_2$, which is such that $\inv_v\beta(P_v)$ is the product of the two coordinates of $\inv_v\boldsymbol{\alpha}(P_v)$.

\subsection{The two-adic place}\label{sec:twoadic}
In the following two lemmas, we compute the values of $\inv_2\boldsymbol{\alpha}$, depending on the congruence type of $m$ modulo $8$. As a result, in both cases, there exists a point $P_2\in\U_{m^2}(\Z_2)$ such that $\inv_2\beta(P_2) = 1$.

\begin{lemma}\label{lem:two35}
Assume $m\equiv 3$ or $5\pmod 8$. Then $\inv_2\beta(P_2)=1$ for every $P_2\in\U_{m^2}(\Z_2)$. Moreover, $x,y,w$ are odd, $w\equiv 3$ or $5\pmod 8$, and we have the following possibilities of $\inv_2\boldsymbol{\alpha}$.
\[
  \renewcommand{\arraystretch}{1.15}
  \begin{array}{c|cc}
     & w\equiv 3 & w\equiv 5\\\hline
    m\equiv 3 & (1,1) & (-1,-1)\\
    m\equiv 5 & (-1,-1) & (1,1)
  \end{array}
\]
\end{lemma}

\begin{proof}
From $m\equiv\pm3 \pmod 8$ we get $m^2\equiv9\pmod{16}$. Fourth powers modulo $16$ are $0$ for even bases and $1$ for odd bases, since an odd square is $1$ or $9$ modulo $16$ and $1^2\equiv9^2\equiv1\pmod{16}$. If $x,y$ were not both odd, then $x^4+y^4\equiv0$ or $1$ modulo $16$, hence $w^2\equiv x^4+y^4-m^2\equiv7$ or $8\pmod{16}$; but squares modulo $16$ lie in $\{0,1,4,9\}$, a contradiction. So $x,y$ are both odd, $w^2\equiv1+1-9\equiv9\pmod{16}$, and among odd residues modulo $8$ only $w\equiv \pm3$ satisfy $w^2\equiv9\pmod{16}$.
Recall that $\epsilon(c)=0$ or $1$ according as $c\equiv1$ or $3\pmod 4$, and $\omega(c)=0$ or $1$ according as $c\equiv\pm1$ or $\pm3\pmod 8$.

Write $s=2^{v_2(s)}a$ and $t=2^{v_2(t)}c$ with $a,c$ odd. Let $A_1=s/(s-t)$ and $A_2=t/(s+t)$, and let $a_i$ be the unit part of $A_i$.
The table below lists, for each pair of residues of $m$ and $w$ modulo $8$, the valuations of $s$, $t$, $s-t$, $s+t$ and the values $\epsilon(a_i)$. A dash indicates that the second entry of the symbol is $\equiv1\pmod 4$, so that $\epsilon(a_i)$ is not required for our computation.
\begin{center}
\renewcommand{\arraystretch}{1.15}
\setlength{\tabcolsep}{5pt}
\begin{tabular}{cc@{\quad}cccc@{\quad}cc@{\quad}c}
\toprule
$m$ & $w$ & $v_2(s)$ & $v_2(t)$ & $v_2(s-t)$ & $v_2(s+t)$ & $\epsilon(a_1)$ & $\epsilon(a_2)$ & $\inv_2\boldsymbol{\alpha}(P_2)$\\
\midrule
$3$ & $5$ & $1$ & $2$ & $1$ & $1$ & $1$ & -- & $(-1,-1)$\\
$3$ & $3$ & $1$ & $1$ & $2$ & $3$ & $1$ & -- & $(1,1)$\\
$5$ & $3$ & $2$ & $1$ & $1$ & $1$ & -- & $1$ & $(-1,-1)$\\
$5$ & $5$ & $2$ & $2$ & $4$ & $3$ & -- & $1$ & $(1,1)$\\
\bottomrule
\end{tabular}
\end{center}
The valuations of $s$ and $t$ follow from $x^2\equiv y^2\equiv1\pmod 8$. In the two rows with $v_2(s)=v_2(t)$, the relation $tu=sr$ reduced modulo $8$ forces $(c,a)\equiv(1,3)$ or $(5,7)\pmod 8$ for $m\equiv3$ and $(c,a)\equiv(1,5),(3,7),(5,1)$ or $(7,3)\pmod 8$ for $m\equiv5$, which determines the valuations of $s\mp t$ and the values $\epsilon(a_i)$. 
In the other two rows $s\mp t$ is twice an odd number. Since $\omega(\pm m)=1$, $\epsilon(m)=1$ for $m\equiv3$ and $\epsilon(-m)=1$ for $m\equiv5$, the two-adic symbol formula gives the last column. In all four cases the two coordinates of $\inv_2\boldsymbol{\alpha}(P_2)$ agree, so $\inv_2\beta(P_2)=\inv_2\alpha_1(P_2)\inv_2\alpha_2(P_2)=1$.
\end{proof}

\begin{lemma}\label{lem:two17}
If $m\equiv 1$ or $7\pmod 8$, there exists $P_2\in\U_{m^2}(\Z_2)$ with $\inv_2\boldsymbol{\alpha}(P_2)=(1,1)$.
\end{lemma}

\begin{proof}
Take $x=y=1$. Since $m^2\equiv1\pmod{16}$, we have $2-m^2\equiv1\pmod{16}$, so $2-m^2$ is a square in $\Z_2$. Choose the square root $w$ with $w\equiv1\pmod 8$. Then $P_2=(1,1,w)\in\U_{m^2}(\Z_2)$, with $s=1-m$, $t=w-1$ and $v_2(t)\geq3$.

Suppose that $m\equiv7\pmod 8$. Then $v_2(s)=1$, so $A_1=(1-t/s)^{-1}\in1+4\Z_2$ and $(A_1,m)_2=1$. Moreover, $-m\equiv1\pmod 8$ is a square in $\Z_2$, hence $(A_2,-m)_2=1$.

Suppose that $m\equiv1\pmod 8$. Then $m$ is a square in $\Z_2$, so $(A_1,m)_2=1$. Writing $m=1+8a$ and $w=1+8b$, the relation $w^2=2-m^2$ gives $a\equiv b\pmod 2$. Put $\theta=t/s=(m+1)/(w+1)$. Then $\theta\equiv1\pmod 8$, so $A_2=\theta/(1+\theta)$ has $v_2(A_2)=-1$ and unit part $\equiv1\pmod 4$. Since $\omega(-m)=0$, we get $(A_2,-m)_2=1$.

In both cases $\inv_2\boldsymbol{\alpha}(P_2)=(1,1)$.
\end{proof}

\subsection{Odd primes dividing \texorpdfstring{$m$}{m}}\label{sec:oddq}
Contrary to the $2$--adic case, for the odd primes, the value of $\inv_q\beta$ can be $-1$, depending on the congruence of $q$ modulo $8$.
The following lemma shows that, in particular, when $q \equiv 1$ or $q \equiv 7 \pmod 8$, then $\inv_q\beta$ is $1$ at every $\Z_q$--point, while for $q \equiv 3 \pmod 8$ we have $\inv_q\beta(P_q)=-1$ for every $P_q\in\U_{m^2}(\Z_q)$.

For $q\mid m$ write
\[
  E_q:=\inv_q\boldsymbol{\alpha}\bigl(\U_{m^2}(\Z_q)\bigr)\subseteq\{\pm1\}^2
\]
for the image of $\inv_q\boldsymbol{\alpha}$.

\begin{lemma}\label{lem:oddq}
Let $q\mid m$ be an odd prime.
\begin{enumerate}
\item If $q\equiv 1\pmod 8$, then $(1,1)\in E_q\subseteq\{(1,1),(-1,-1)\}$.
\item If $q\equiv 3\pmod 8$, then $E_q=\{(-1,1),(1,-1)\}$.
\item If $q\equiv 5\pmod 8$, then $E_q=\{\pm1\}^2$.
\item If $q\equiv 7\pmod 8$, then $(1,1)\in E_q\subseteq\{(1,1),(-1,-1)\}$.
\end{enumerate}
\end{lemma}

\begin{proof}
Write $m=qM$ with $M\in\Z_q^\times$, and let $A_1=s/(s-t)$, $A_2=t/(s+t)$.
To compute the image of $\U_{m^2}(\Z_q)$, we distinguish four cases, according to which of the coordinates $x$ and $y$ of $P_q = (x,y,w)$ are divisible by $q$.

\underline{Case A: $q\mid x$, $q\nmid y$.} Then $w^2\equiv y^4\pmod{q^2}$, so $w\equiv\pm y^2\pmod q$, and $v_q(s)=v_q(r)=1$ with $s/q\equiv-M$, $r/q\equiv M\pmod q$. In the case $w\equiv y^2$, $u$ is a unit, so $tu=sr$ gives $v_q(t)=2$ with $t/q^2\equiv-M^2/(2y^2)\pmod q$. Hence $A_1\equiv1\pmod q$, while $v_q(A_2)=1$ with unit part $a_0\equiv M/(2y^2)$, and the Hilbert symbol formula gives
\[
  (A_2,-m)_q=(-1)^{\frac{q-1}{2}}\legendre{a_0}{q}\legendre{-M}{q}
  =(-1)^{\frac{q-1}{2}}\legendre{-1}{q}\legendre{2}{q}=\legendre{2}{q}.
\]
If $w\equiv-y^2$, then $A_1 \equiv 1 \pmod q$ and the unit part of $A_1$ is congruent to $-\dfrac{M}{2y^2}$. In this case we get $(A_1,m)_q = \legendre{2}{q}$.

Thus $\inv_q\boldsymbol{\alpha}(P_q)=\left(1,\legendre{2}{q}\right)$ if $w \equiv y^2 \pmod q$ and $\inv_q\boldsymbol{\alpha}(P_q)=\left(\legendre{2}{q},1\right)$ if $w \equiv -y^2 \pmod q$.

Choosing the points $P_q = (0,1,w)$, with $w$ such that $w^2=1-m^2$, $w\equiv\pm1\pmod q$, both values are attained.

\underline{Case A': $q\mid y$, $q\nmid x$.} Now $s,t,u,r$ are units and, by $(w-x^2)(w+x^2)=y^4-m^2$, one of $w\pm x^2$ has valuation $2$ and exactly one between $s+t$ and $s-t$ has valuation $1$, while the other is a unit. The same computation as above gives the same two values  $\left(1,\legendre{2}{q}\right)$, $\left(\legendre{2}{q},1\right)$ for $\inv_q\boldsymbol{\alpha}(P_q)$, attained at the points $P_q =(1,0,w)$, with the same conditions on $w$ as above.

These two cases show that, from the property of $\legendre{2}{q}$, if $q\equiv\pm1 \pmod 8$ then $(1,1) \in E_q$, and if $q\equiv\pm3\pmod 8$, then both $(1,-1)$ and $(-1,1)$ are contained in $E_q$.
To complete the proof of our statements, we compute the values of $\inv_q\beta(P_q)$ in the remaining two cases.

\underline{Case B: $q\nmid xy$.} Then $s,t,s\pm t$ are units. So, by using $s\equiv x^2$ and $x^4-t^2\equiv2y^2t\pmod q$, we have that 
$\inv_q\alpha_i(P_q)=\legendre{A_i}{q}$ and
\[
  \inv_q\beta(P_q)=\legendre{A_1A_2}{q},\qquad A_1A_2=\frac{st}{s^2-t^2}\equiv\frac{x^2}{2y^2}\pmod q.
\]
Hence $\inv_q\beta(P_q)=\legendre{2}{q}$.

\underline{Case C: $q\mid x$, $q\mid y$.} In this case, $w=qW$ with $W\in\Z_q^\times$ and $W^2\equiv-M^2\pmod q$, so $-1$ is a square modulo $q$ and $q\equiv1\pmod 4$. Fix $\iota\in\Z_q$ with $\iota^2=-1$ and $W\equiv \pm \iota M\pmod q$. Then $v_q(s)=v_q(t)=1$ with $s/q\equiv-M$, $t/q\equiv \pm\iota M$, and $A_1\equiv A_2\equiv(1\pm\iota)^{-1}\pmod q$, so both coordinates of $\inv_q\boldsymbol{\alpha}(P_q)$ equal $\legendre{1\pm\iota}{q}$. In particular, this implies that $\inv_q\beta(P_q) = 1$.

To conclude the proof, notice that case C occurs only when $q \equiv 1$ or $q \equiv 5 \pmod8$. 
So if $q \equiv 3 \pmod 8$, then $\inv_q\beta(P_q) = -1$ for all $P_q$ implies that $E_q=\{(-1,1),(1,-1)\}$, and if $q \equiv 7 \pmod 8$, then $E_q\subseteq\{(1,1),(-1,-1)\}$.
For $q \equiv 1 \pmod 8$, we have that $E_q\subseteq\{(1,1),(-1,-1)\}$.
On the other hand, when $q \equiv 5 \pmod 8$, for case C the points $P_q = (0,0,\pm m\iota)$ give the values $\inv_q\alpha_i(P_q) = \legendre{1+\iota}{q}$ and $\inv_q\alpha_i(P_q) = \legendre{1-\iota}{q}$, respectively. Since the product of these two is $\legendre{2}{q}$, we deduce that $(1,1)$ and $(-1,-1)$ are both contained in $E_q$.  
\end{proof}

\section{The integral Brauer--Manin obstruction}\label{sec:proof}
In this section, we prove Theorem~\ref{thm:DP-BMO}. We start by showing that $\U_{m^2}$ is everywhere locally soluble.

\begin{lemma}\label{lem:local-sol}
For every odd integer $m>0$, the surface $U_{m^2}$ is everywhere locally soluble; that is, $\U_{m^2}(\A_\Z)\neq\emptyset$.
\end{lemma}

\begin{proof}
At the real place one may take $y=0$ and $x$ sufficiently large. At $2$, take $x=y=1$: since $2-m^2\equiv1\pmod 8$, it is a square in $\Z_2$. If an odd prime $q$ divides $m$, take $x=0$, $y=1$; the unit $1-m^2\equiv1\pmod {q^2}$ is a square in $\Z_q$.

It remains to consider an odd prime $p\nmid m$. The equation defining $X_{m^2}$ gives a smooth degree-$2$ del Pezzo surface over $\F_p$, and $D$ has smooth reduction. The Lefschetz trace formula and the Weil bound give
\[
 \#X_{m^2}(\F_p)\geq p^2-8p+1,
 \qquad
 \#D(\F_p)\leq p+1+2\sqrt p.
\]
Consequently $U_{m^2}(\F_p)\neq\emptyset$ for $p\geq11$. The remaining primes are immediate: for $p=3$ use $(x,y,w)=(1,0,0)$; for $p=5$, use $(1,0,0)$ if $m^2\equiv1 \pmod 5$ and $(0,0,1)$ if $m^2\equiv4 \pmod 5$; and for $p=7$, use $(1,0,0)$, $(1,1,0)$, or $(3,0,0)$ according as $m^2\equiv1,2$, or $4$ modulo $7$. All these points are smooth, so Hensel's lemma gives a point over $\Z_p$.
\end{proof}
Recall that by Theorem~\ref{thm:DP-affine-BG} we have $\Br U_{m^2}/\Br\Q=\langle\alpha_1,\alpha_2\rangle$, so
\begin{equation}\label{eq:bm-inclusion}
  \U_{m^2}(\A_\Z)^{\Br}=\U_{m^2}(\A_\Z)^{\langle\alpha_1,\alpha_2\rangle}\subseteq\U_{m^2}(\A_\Z)^{\beta},
\end{equation}
where the inclusion follows from $\beta=\alpha_1+\alpha_2\in\langle\alpha_1,\alpha_2\rangle$.
We use this fact by showing that if $\omega_5(m)=0$ and $\omega_3(m)$ is odd then there is no integral adelic point orthogonal to $\beta$. On the other hand, when this assumption does not hold, we show that $\U_{m^2}(\A_\Z)^{\Br}$ is not empty.

\begin{proof}[{Proof of Theorem~\ref{thm:DP-BMO}}]
    Assume that $\omega_5(m)=0$ and $\omega_3(m)$ is odd. We show that $\U_{m^2}(\A_\Z)^{\beta}=\emptyset$.
    By Lemma~\ref{lem:local-sol}, $\U_{m^2}(\A_\Z)\neq\emptyset$ and choose any $(P_v)_v\in\U_{m^2}(\A_\Z)$. By Lemma~\ref{lem:easy} only $2$ and the primes $q\mid m$ contribute to $\prod_v\inv_v\beta(P_v)$.
    Since $\omega_5(m)=0$, one has $m\equiv3^{\omega_3(m)}\cdot(\pm1)\pmod 8$, which is $\equiv3$ or $5\pmod 8$ as $\omega_3(m)$ is odd; so Lemma~\ref{lem:two35} gives $\inv_2\beta(P_2)=1$. By Lemma~\ref{lem:oddq}, each $q\mid m$ contributes $-1$ if and only if $q\equiv3\pmod 8$. Hence
    \[
      \prod_v\inv_v\beta(P_v)=(-1)^{\omega_3(m)}=-1,
    \]
    so no integral adelic point is orthogonal to $\beta$.

    For the other implication, we show that if $m$ is such that $\omega_5(m)>0$ or $\omega_3(m)$ is even, then $\U_{m^2}(\A_\Z)^{\langle\alpha_1,\alpha_2\rangle}$ is not empty.
    Suppose first $\omega_5(m)>0$, and fix a prime $q_0\mid m$ with $q_0\equiv5\pmod 8$. 
    By Lemma~\ref{lem:oddq}, $E_{q_0} = \{\pm1\}^2$, so we can choose a suitable $P_{q_0}$ such that $\prod_v\inv_v \boldsymbol{\alpha}(P_v) = (1,1)$, for any other choice of $P_v$ with $v \neq q_0$.

    Now suppose $\omega_5(m)=0$ and $\omega_3(m)$ even, so $m\equiv1$ or $7\pmod 8$.
    By Lemma~\ref{lem:two17}, there exists $P_2$ with $\inv_2\boldsymbol{\alpha}(P_2)=(1,1)$. 
    For each $q\mid m$ with $q\equiv1$ or $7\pmod 8$ choose $P_q$ with $\inv_q\boldsymbol{\alpha}(P_q)=(1,1)$, which exists by Lemma~\ref{lem:oddq}.
    Again by Lemma~\ref{lem:oddq}, we can choose $P_q$ such that $\inv_q\boldsymbol{\alpha}(P_q)=(1,-1)$ (or $(-1,1)$) for every $q \mid m$ with $q \equiv 3 \pmod 8$. Since there are an even number of such $q$'s, and since all other places contribute $(1,1)$ by Lemma~\ref{lem:easy}, the resulting adelic point is orthogonal to both $\alpha_1$ and $\alpha_2$.
\end{proof}

\begin{remark}
    For $m$ an odd prime, $\omega_3(m),\omega_5(m)\in\{0,1\}$. Thus from Theorem~\ref{thm:DP-BMO} there is a Brauer--Manin obstruction if and only if $m \equiv 3 \pmod 8$.
\end{remark}

\begin{remark}\label{rem:177}
For the values of $m$ for which there is no integral Brauer--Manin obstruction, we cannot exclude that some other type of obstruction holds.
For example, the value $m=177=3\cdot 59$ has $\omega_3(177)=2$ and $\omega_5(177)=0$, so by Theorems~\ref{thm:DP-affine-BG} and~\ref{thm:DP-BMO} there is no integral Brauer--Manin obstruction on $U_{177^2}$.
However, an exhaustive search done with both \texttt{Magma} and SageMath finds no integral point with $\max\{|x|,|y|\}\leq 50,000,000$.
On the other hand, for the values $m=57,129,201,209$, which have the same congruence pattern, we found integral points of height at most $10^4$.
We also checked that, for each of these $m$'s, at least one of the points satisfies the sign condition of \cite[Thm.~1.1]{alessandriloughran}, implying that $U_{m^2}(\Z)$ is not thin.

This suggests that either $U_{177^2}(\Z)=\emptyset$ and so the failure of the integral Hasse principle might be explained by another obstruction, unrelated to the congruences of $m$ modulo $8$, or there is some integral point, not detected by our search. In the latter case, it would suffice to find one integral point satisfying the same sign condition to conclude that $U_{177^2}(\Z)$ is not thin, as for the other values of $m$.
\end{remark}

\section{The density of parameters with an obstruction}\label{sec:counting}

Theorem~\ref{thm:DP-BMO} describes the odd squarefree $m$ with an integral Brauer--Manin obstruction by a condition on the residues modulo $8$ of the prime factors of $m$.
Thanks to this characterization, we are able to count the number of affine surfaces $\U_{m^2}$ such that there is an integral Brauer--Manin obstruction to the Hasse principle.
\begin{corollary}\label{cor:counting}
    For $T\geq 2$ let 
    \[N_{\Br}(T):=\#\{1<m\leq T:\ m\ \text{odd, squarefree},\ \U_{m^2}(\A_\Z)^{\Br}=\emptyset\}.\]
     Then
    \begin{equation*}
    N_{\Br}(T)=c\,
    \frac{T}{(\log T)^{1/4}}\left(1+O\bigl((\log T)^{-1/2}\bigr)\right),
    \end{equation*}
    where
    \[
    c=\frac{1}{2\,\Gamma(3/4)}\lim_{s\to1^+}(s-1)^{3/4}\prod_{\substack{p\ \rm{prime}\\ p\equiv1,3,7\bmod 8}}\left(1+\frac{1}{p^{s}}\right),
    \]
    the limit exists and $c>0$. 
\end{corollary}
The argument below is inspired by \cite[Lemma~5.19]{LoughranMitankin2021}, where Loughran and Mitankin count the Markoff surfaces with an integral Brauer--Manin obstruction; the proof is a standard application of a Tauberian theorem.

\begin{proof}
By Theorem~\ref{thm:DP-BMO}, we count odd squarefree
integers $m\leq T$ with $\omega_5(m)=0$ and
$\omega_3(m)$ odd. Let $f$ be the indicator function
of the odd squarefree integers with no prime factor
congruent to $5$ modulo $8$, and put
$g(n)=(-1)^{\omega_3(n)}f(n)$. Then
\[
  N_{\Br}(T)=\frac12\sum_{n\leq T}\bigl(f(n)-g(n)\bigr).
\]
The associated Dirichlet series have Euler products
\[
\begin{aligned}
  F(s)&=\prod_{p\equiv1,3,7\ \modn 8}\left(1+\frac{1}{p^s}\right),\\
  G(s)&=\prod_{p\equiv1,7\ \modn 8}\left(1+\frac{1}{p^s}\right)
         \prod_{p\equiv3\ \modn 8}\left(1-\frac{1}{p^s}\right).
\end{aligned}
\]
By character orthogonality modulo $8$, we may write
\begin{equation}\label{eq-Dseries}    F(s)=\zeta(s)^{3/4}H_f(s),
  \qquad
  G(s)=\zeta(s)^{1/4}H_g(s),
\end{equation}
where $H_f$ and $H_g$ are products of powers of
non-principal Dirichlet \(L\)-functions modulo \(8\)
and Euler products converging absolutely on
\(\operatorname{Re}s>1/2\).
The Dirichlet series counting the required integers is 
\[
  D(s) = \frac{F(s)-G(s)}{2}.
\]
Its coefficients are non-negative. 
Moreover, it admits analytic continuation across the line $\operatorname{Re}(s)=1$, away from $s=1$. According to \eqref{eq-Dseries}, its leading singular term comes from $F(s)/2$. Therefore,
\[ 
 D(s) \sim \frac{H_f(1)}{2(s-1)^{3/4}}.
\]
A Tauberian theorem \cite[Thm.~II.7.28]{Tenenbaum2015} gives 
\[
  N_{\Br}(T)\sim\frac{H_f(1)}{2\Gamma(3/4)}
      \frac{T}{(\log T)^{1/4}}
\]
Furthermore, 
\[  
H_f(1) =  \lim_{s\to1^+}(s-1)^{3/4}F(s) > 0
\]
which gives the constant $c$ in the statement.

\end{proof}

\bibliographystyle{alpha}
\bibliography{ref}
\end{document}